\documentclass[12pt,reqno]{article}
\usepackage[usenames]{color}
\usepackage{amssymb}
\usepackage{amsmath}
\usepackage{amsthm}
\usepackage{amsfonts}
\usepackage{amscd}
\usepackage{graphicx}

\usepackage{pdfpages}

\usepackage[colorlinks=true,
linkcolor=webgreen,
filecolor=webbrown,
citecolor=webgreen]{hyperref}

\definecolor{webgreen}{rgb}{0,.5,0}
\definecolor{webbrown}{rgb}{.6,0,0}

\usepackage{color}
\usepackage{fullpage}
\usepackage{float}
\usepackage{graphics}
\usepackage{latexsym}
\usepackage{epsf}
\usepackage{breakurl}

\usepackage{listings}
\usepackage{xcolor}

\newcommand{\seqnum}[1]{\href{https://oeis.org/#1}{\rm \underline{#1}}}

\theoremstyle{plain}
\newtheorem{theorem}{Theorem}

\theoremstyle{definition}

\newtheorem{conjecture}[theorem]{Conjecture}
\theoremstyle{remark}
\newtheorem{remark}[theorem]{Remark}

\begin{document}

\begin{center}
\vskip 1cm{\Large\bf
On Generalized Eigenvalues of MAX Matrices to MIN Matrices and of LCM Matrices
to GCD Matrices
}
\vskip 1cm

\begin{minipage}{3in}
\begin{center}
Jorma K. Merikoski\\
Faculty of Information Technology and Communication Sciences \\
FI-33014 Tampere University\\
Finland \\
\href{mailto:jorma.merikoski@tuni.fi}{\tt jorma.merikoski@tuni.fi}\\
\end{center}
\end{minipage}\ \quad \
\begin{minipage}{3in}
 \begin{center}
 Pentti Haukkanen\\
Faculty of Information Technology and Communication Sciences \\
FI-33014 Tampere University\\
Finland \\
\href{mailto:pentti.haukkanen@tuni.fi}{\tt pentti.haukkanen@tuni.fi}\\
 \end{center}
\end{minipage}
\vskip .15in

\begin{minipage}{3in}
\begin{center}
 Antonio Sasaki \\
Centre de Math\'ematiques Appliqu\'ees \\
\'Ecole nationale sup\'erieure des mines de Paris \\
Universit\'e Paris Sciences et Lettres\\
FR-06560 Valbonne \\
France \\
\href{mailto:antonio.sasaki@minesparis.psl.eu}{\tt antonio.sasaki@minesparis.psl.eu} \\
\mbox{}\\
\end{center}
\end{minipage}\ \quad \
\begin{minipage}{3in}
\begin{center}
Timo Tossavainen \\
Department of Health, Education and Technology \\
Lulea University of Technology \\
SE-97187 Luleå\\
Sweden \\
\href{mailto:timo.tossavainen@ltu.se}{\tt timo.tossavainen@ltu.se} \\
 \end{center}
\end{minipage}
\vskip .15in
\end{center}

\begin{abstract}
We determine, for any $n \geq 1$, the generalized eigenvalues of an $n \times n$ MAX matrix to the corresponding MIN matrix. We also show that a similar result holds for the generalized eigenvalues of an $n \times n$ LCM matrix to the corresponding GCD matrix when $n\le4$, but breaks down for $n>4$. In addition, we prove Cauchy's interlacing theorem for generalized eigenvalues, and we conjecture an unexpected connection between the OEIS sequence \seqnum{A004754} and the appearance of $-1$ as a generalized eigenvalue in the LCM--GCD setting.
\end{abstract}

\section{Introduction}

Let ${\bf A}$ and ${\bf B}$ be complex Hermitian $n\times n$ matrices, and let~${\bf B}$ be
positive definite. (That is, the conjugate transpose ${\bf A}^*=\bf A$, and
${\bf x}^*{\bf B}{\bf x}>0$ whenever ${\bf 0}\ne{\bf x}\in\mathbb{C}^n$.) The generalized eigenvalue equation of~$\bf A$ to~$\bf B$ is
\begin{equation}
\label{gevp}
{\bf Ax}=\lambda{\bf Bx},\quad{\bf 0}\ne{\bf x}\in\mathbb{C}^n.
\end{equation}
Then $\lambda$ is a generalized eigenvalue (``g-eigenvalue'' for short) of ${\bf A}$  to ${\bf B}$,
and ${\bf x}$ is a corresponding generalized eigenvector (``g-eigenvector'').
See, e.g., Ghojogh, Karray, and Crowley~\cite{GKC}.
They consider real symmetric matrices, but everything can be generalized to complex
Hermitian matrices.

It is actually enough that $\bf B$ is invertible in~(\ref{gevp}), and $\bf A$ can be arbitrary.
However, the above assumptions are usually stated. Then all g-eigenvalues are real,
and g-eigenvectors corresponding to distinct g-eigenvalues are orthogonal with respect
to the inner product $\langle{\bf x},{\bf y}\rangle={\bf y}^*\bf Bx$.

The standard eigenvalues (``s-eigenvalues'' for short) are widely studied. The g-eigenvalue equation~(\ref{gevp})
reduces to the s-eigenvalue equation
\[
{\bf B}^{-1}{\bf Ax}=\lambda{\bf x}.
\]
However, this ``quick and dirty solution''~\cite{GKC} does not have significant use.
So, g-eigenvalues must be considered in a different way. 
This area has not been studied much in the literature.

Let
\[
S=\{s_1,\dots,s_n\},\quad s_1<\cdots<s_n,
\]
be a set of positive real numbers.
The $n \times n$ MAX matrix~${\bf M}_S$ and MIN matrix~${\bf N}_S$ on~$S$ are defined by
\[
{\bf M}_S=(m^S_{ij}), \ m^S_{ij}=\max{(s_i,s_j)},\quad{\bf N}_S=(n^S_{ij}), \ n^S_{ij}=\min{(s_i,s_j)}.
\]
Then ${\bf N}_S$ is positive definite \cite[Theorem 8.1]{MH}. Also, let
\[
T=\{t_1,\dots,t_n\},\quad t_1<\cdots<t_n,
\]
be a set of positive integers. The $n \times n$ LCM matrix~${\bf L}_T$ and GCD
matrix~${\bf G}_T$ on~$T$ are defined by
\[
{\bf L}_T=(l^T_{ij}), \ l^T_{ij}=\mathrm{lcm}{(t_i,t_j)},\quad{\bf G}_T=(g^T_{ij}), \ g^T_{ij}=\gcd{(t_i,t_j)}.
\]
Also, ${\bf G}_T$ is positive definite~\cite[Theorem~2]{BL}.

We study g-eigenvalues of~${\bf M}_S$ to~${\bf N}_S$ in Section~\ref{max}, and those
of~${\bf L}_T$ to~${\bf G}_T$ in Sections \ref{lcm} and~\ref{lcmt}. Finally, we complete our
paper with discussion in Section~\ref{discussion}.

All g-eigenvalues of $\bf A$ to~$\bf A$ are trivially one.
We can therefore expect that
the g-eigenvalues of $\bf A$ to~$\bf B$ are also in certain nontrivial cases
more accessible than the s-eigenvalues of $\bf A$ and~$\bf B$. We will see this
in the case ${\bf A}={\bf M}_S$, ${\bf B}={\bf N}_S$. We will also see it in the case
${\bf A}={\bf L}_T$, ${\bf B}={\bf G}_T$, where $T=\{1,\dots,n\}$, $n\le 4$.

Recently, these matrices have been studied extensively (e.g., \cite{AFKT,FKT,HP,Lo,MH,ZWF}).
These works discuss not only new results in this field but also provide applications to various
other areas of mathematics. Applications in computing~\cite{KMR}, statistics~\cite{MH},
and signal processing~\cite{RRBA} have also been reported.

\section{MAX--MIN setting}
\label{max}

We want to evaluate the g-eigenvalues of ${\bf M}_S$ to ${\bf N}_S$, i.e., the
solutions~$\lambda$ to the equation $\det({\bf M}_S - \lambda {\bf N}_S) = 0$.

We begin with $n=2$. Let $S=\{a, b\}$, $0<a < b$. Then
\[
\det{({\bf M}_S-\lambda{\bf N}_S)}=\left|
\begin{array}{cc}
a-\lambda a&b-\lambda a
\\
b-\lambda a&b-\lambda b
\end{array}
\right|=a(b-a)\lambda^2-b(b-a)=0
\]
if and only if
\[
\lambda=\pm\sqrt{\frac{b}{a}}.
\]

Our aim is to prove Theorem \ref{mnthm} below. However, because the general proof is not 
easily readable, we show the details only in the case $n=4$. A careful reader will notice that we
can proceed similarly for any $n>2$. 

\begin{theorem}
\label{mnthm}
The g-eigenvalues of ${\bf M}_S$ to~${\bf N}_S$, $n>2$, are
\begin{equation}
\label{lambdas}
\lambda_1=\sqrt{\frac{s_n}{s_1}}, \ \lambda_2=\cdots=\lambda_{n-1}=-1, \ \lambda_n=-\sqrt{\frac{s_n}{s_1}}.
\end{equation}
\end{theorem}

\begin{proof}
 Let $S=\{a,b,c,d\}$, $0<a<b<c<d$. Then
\[
{\bf M}_S=
\left(
\begin{array}{cccc}
a&b&c&d
\\
b&b&c&d
\\
c&c&c&d
\\
d&d&d&d
\end{array}
\right),\quad
{\bf N}_S=
\left(
\begin{array}{cccc}
a&a&a&a
\\
a&b&b&b
\\
a&b&c&c
\\
a&b&c&d
\end{array}
\right).
\]
The matrix
\[
{\bf M}_S+{\bf N}_S=\left(
\begin{array}{cccc}
2a&a+b&a+c&a+d
\\
a+b&2b&b+c&b+d
\\
a+c&b+c&2c&c+d
\\
a+d&b+d&c+d&2d
\end{array}
\right)
\]
has rank~2 and nullity~2. Consequently,  $-1$ is a g-eigenvalue of~${\bf M}_S$
to~${\bf N}_S$ with multiplicity~2. We show that the remaining g-eigenvalues are
\[
\lambda=\pm\sqrt{\frac{d}{a}}.
\]
Regardless of the sign of $\lambda$, we have (note that $d=\lambda^2a$)
\[
\det{({\bf M}_S+\lambda{\bf N}_S)}=
\left|
\begin{array}{cccc}
a+\lambda a&b+\lambda a&c+\lambda a&d+\lambda a
\\
b+\lambda a&b+\lambda b&c+\lambda b&d+\lambda b
\\
c+\lambda a&c+\lambda b&c+\lambda c&d+\lambda c
\\
d+\lambda a&d+\lambda b&d+\lambda c&d+\lambda d
\end{array}
\right|
\]
\[
=
\left|
\begin{array}{cccc}
a+\lambda a&b+\lambda a&c+\lambda a&d+\lambda a
\\
b-a&\lambda(b-a)&\lambda(b-a)&\lambda(b-a)
\\
c-b&c-b&\lambda(c-b)&\lambda(c-b)
\\
d-c&d-c&d-c&\lambda(d-c)
\end{array}
\right|
=\left|
\begin{array}{cccc}
\lambda a&b&c&\lambda^2a
\\
b-a&\lambda(b-a)&\lambda(b-a)&\lambda(b-a)
\\
c-b&c-b&\lambda(c-b)&\lambda(c-b)
\\
d-c&d-c&d-c&\lambda(d-c)
\end{array}
\right|
\]
\[
+
\left|
\begin{array}{cccc}
a&\lambda a&\lambda a&\lambda a
\\
b-a&\lambda(b-a)&\lambda(b-a)&\lambda(b-a)
\\
c-b&c-b&\lambda(c-b)&\lambda(c-b)
\\
d-c&d-c&d-c&\lambda(d-c)
\end{array}
\right|=:D_1+D_2.
\]
Since $D_1=D_2=0$, the claim follows.
\end{proof}

\begin{remark}
Actually Theorem~$\ref{mnthm}$ holds also for $n=2$. Then the equation chain
$\lambda_2=\cdots=\lambda_{n-1}$ is ``empty''.
\end{remark}

\begin{remark}
\label{remark}
If the ordering of $s_1,\dots,s_n$ is arbitrary, then (\ref{lambdas}) reads
\begin{equation}
\label{lmbds}
\lambda_1=\max_{i,j}\sqrt{\frac{s_i}{s_j}}, \ \lambda_2=\cdots=\lambda_{n-1}=-1,
 \ \lambda_n=-\max_{i,j}\sqrt{\frac{s_i}{s_j}}.
\end{equation}
See also \cite[Remark~2.1]{MH}.
\end{remark}

\begin{remark}
Theorem~\ref{mnthm} applies also to the g-eigenvalues of~${\bf M}_S$
to~${\bf N}_{S'}$, where
\[
S'=\{s_1',\dots,s_n'\},\quad s_1'-s_1=\cdots=s_n'-s_n.
\]
\end{remark}

\section{LCM--GCD setting on $T=\{1,2,\dots,n\}$}
\label{lcm}

\subsection{The case $n\le 4$}
\label{nle4}

Let $T=\{1,2,\dots,n\}$, let $\lambda_{n1}\ge\cdots\ge\lambda_{nn}$
be the g-eigenvalues of ${\bf L}_T$ to~${\bf G}_T$, and let
$p_n(\lambda)=\det{({\bf L}_T-\lambda{\bf G}_T)}$
be the g-characteristic polynomial. Then
\allowdisplaybreaks
\begin{align*}
p_1(\lambda)&=1-\lambda,\quad\lambda_{11}= 1,
\\
p_2(\lambda)&=
\left|
\begin{array}{cc}
1-\lambda&2-\lambda
\\
2-\lambda&2-2\lambda
\end{array}
\right|
=\lambda^2-2,\quad\lambda_{21}=\sqrt{2}, \ \lambda_{22}=-\sqrt{2},
\\
p_3(\lambda)&=
\left|
\begin{array}{ccc}
1-\lambda & 2-\lambda & 3-\lambda\\
2-\lambda & 2-2\lambda & 6-\lambda\\
3-\lambda & 6-\lambda & 3-3\lambda
\end{array}
\right|
=-2(\lambda+1)(\lambda^2-6),
\quad\lambda_{31}=\sqrt{6}, \ \lambda_{32}=-1, \ \lambda_{33}=-\sqrt{6},
\\
p_4(\lambda)&=
\left|
\begin{array}{cccc}
1-\lambda & 2-\lambda & 3-\lambda & 4-\lambda \\
2-\lambda & 2-2\lambda & 6-\lambda & 4-2\lambda \\
3-\lambda & 6-\lambda & 3-3\lambda & 12-\lambda\\
4-\lambda & 4-2\lambda & 12-\lambda & 4-4\lambda
\end{array}
\right|
=4(\lambda + 1)^2(\lambda^2 - 12),
\\
&\quad\lambda_{41}= \sqrt{12}, \ \lambda_{42}=\lambda_{43}=-1, \ \lambda_{44}=-\sqrt{12}.
\end{align*}

\subsection{The case $n>4$}

The g-eigenvalues in Section \ref{nle4} suggest that there may also be values of $n>4$ such that
\begin{equation}
\label{expect}
\lambda_{n1}=\sqrt{m}, \ \lambda_{n2}=\cdots=\lambda_{n,n-1}=-1, \ \lambda_{nn}=-\sqrt{m}
\end{equation}
for some integer $m$.
We examine this hypothesis and begin with the case  $n=5$. 
We have
\begin{align*}
p_5(\lambda)&=-16\lambda^5-48\lambda^4+528\lambda^3+2480\lambda^2+2880\lambda+960
\\
&=-16(\lambda+1)(\lambda^4+2\lambda^3-35\lambda^2-120\lambda-60) \\
&=:
-16(\lambda+1)q(\lambda).
\end{align*}
Because $q(-1)=24\ne 0$, the multiplicity of $\lambda=-1$ is only one, falsifying~(\ref{expect}).
The g-eigenvalues are
\[
\lambda_{51}=6.4798, \ \lambda_{52}=-0.6118, \ \lambda_{53}=-1, \ \lambda_{54}=-3.3489, \
\lambda_{55}=-4.5191.
\]
(These are approximations to four decimal places, similarly throughout the paper.)
Interestingly, $\sqrt{42}=6.4807$ is near to~$\lambda_{51}$.
If their difference were due to rounding errors, then the first equation in~(\ref{expect}) would hold
for $n=5$, too. But $p_5(\sqrt{42})=-3168\sqrt{42}+20448$, showing that the difference is actual.

Moreover,
\[
\lambda_{61}=6.8501, \ \lambda_{62}=2.5592, \ \lambda_{63}=-0.7419, \ \lambda_{64}=-1.3749, \
\lambda_{65}=-3.4396, \ \lambda_{66}=-5.8528.
\]
Thus $-1$ is not a g-eigenvalue when $n=6$.  

These two cases already 
suffice to make it fairly clear that
$\lambda_{n2}=\cdots=\lambda_{n,n-1}=-1$ does not hold for any $n>4$. However, we choose
to verify this claim thoroughly, as it involves first proving and then applying Cauchy's interlacing
theorem \cite[Theorem 4.3.17]{HJ} for g-eigenvalues -- a result that is arguably of general interest.
To that end, recall that there are two positive
g-eigenvalues for $n=6$. By Theorem~\ref{cauchy} below, there are at least two positive
g-eigenvalues for $n=7$. Continuing in this way confirms the claim.

We conclude this section by exploring in which dimensions $-1$ occurs as a
g-eigenvalue. Computer experiments covering the range $1\le n\le 1000$ (with code provided in the Appendix) show that $-1$ is a g-eigenvalue if and only if
\[
n = \overbrace{4, 5}^{2}, \overbrace{8, 9, 10, 11}^{2^2}, 
\overbrace{16,\ldots, 23}^{2^3}, \overbrace{32,\ldots, 47}^{2^4}, \overbrace{64,\ldots, 95}^{2^5}, \overbrace{128,\ldots, 191}^{2^6}, \overbrace{256,\ldots, 383}^{2^7}, 512,\ldots,
\]
where the overbrace indicates the number of terms. This sequence is the same as the
OEIS~\cite{oeis} sequence \seqnum{A004754} without the first term. Its description~\cite{oeis} raises an interesting conjecture.

\begin{conjecture}
\label{conj}
Let $T=\{1,\dots,n\}$, $n>3$.
Then $-1$ is a g-eigenvalue of ${\bf L}_T$ to~${\bf G}_T$ if and only
if the binary representation of $n$ begins with $10$.
\end{conjecture}
\noindent
For example, $4=(100)_2$, $5=(101)_2$, $8=(1000)_2$, $19=(10011)_2$.

This OEIS sequence~$(a_n)$ satisfies~\cite{oeis}
\begin{equation}
\label{formula}
a_{2^m+k}=2^{m+1}+k,\quad m\ge 0, \ 0\le k<2^m.
\end{equation}
If, for example, $m=k=3$, then the left-hand side equals $a_{8+3}=a_{11}=19$, and
the right-hand side equals $16+3=19$.
An induction proof of Conjecture~\ref{conj} can perhaps be found by using~(\ref{formula}).

\subsection{Cauchy's interlacing theorem for g-eigenvalues}

\begin{theorem}
\label{cauchy}
Let $\bf A$ and $\bf B$ be as in~$(\ref{gevp})$, $n>1$, with first leading principal submatrices ${\bf A}'$
and respectively~${\bf B}'$ $($obtained by removing the $n$th row and column$)$.
Let
\[
\lambda_1\ge\cdots\ge\lambda_n\quad{\it and}\quad\lambda_1'\ge\cdots\ge\lambda_{n-1}'
\]
be the g-eigenvalues of $\bf A$ to~$\bf B$ and, respectively, of ${\bf A}'$ to~${\bf B}'$.
Then
\[
\lambda_1\ge\lambda_1'\ge\lambda_2\ge\lambda_2'\ge\cdots\ge\lambda_{n-1}\ge\lambda_{n-1}'\ge\lambda_n.
\]
\end{theorem}
\begin{proof}
Let $\preceq$ denote the subspace inclusion. Because the Courant--Fischer theorem
\cite[Theorem 4.2.6]{HJ} extends to g-eigenvalues \cite[Theorem~3]{ANT} 
(note the wrong ordering of max and min in its formulation), we have
\begin{equation}
\label{maxmin}
\lambda_k=\max_{\substack{U\preceq\mathbb{C}^n\\\dim{U}=k}}\min_{{\bf 0}\ne{\bf x}\in U}
\frac{{\bf x}^*{\bf Ax}}{{\bf x}^*{\bf Bx}},\quad k=1,\dots,n,
\end{equation}
and
\begin{equation}
\label{mxmn}
\lambda_k'=\max_{\substack{V\preceq\mathbb{C}^{n-1}\\\dim{V}=k}}\min_{{\bf 0}\ne{\bf y}\in V}
\frac{{\bf y}^*{\bf A}'{\bf y}}{{\bf y}^*{\bf B}'{\bf y}},\quad k=1,\dots,n-1.
\end{equation}
Let
\[
{\bf 0}\ne{\bf x}\in\mathbb{C}^n,\quad
{\bf x}=\left(\begin{array}{c}
{\bf x}'\\x_n\end{array}\right),\quad
{\bf A}=\left(
\begin{array}{cc}
{\bf A}'&{\bf u}
\\
{\bf u}^*&a_{nn}
\end{array}
\right),\quad
{\bf B}=\left(
\begin{array}{cc}
{\bf B}'&{\bf v}
\\
{\bf v}^*&b_{nn}
\end{array}
\right).
\]
If $x_n=0$, then
\[
\frac{{\bf x}^*{\bf Ax}}{{\bf x}^*{\bf Bx}}=\frac{({\bf x}')^*{\bf A}'{\bf x}'}{({\bf x}')^*{\bf B}'{\bf x}'},
\]
so, for $k=1,\dots,n-1$, 
\[
\lambda_k'=\max_{\substack{V\preceq\mathbb{C}^{n-1}\\\dim{V}=k}}
\min_{\substack{{\bf 0}\ne{\bf x}'\in V\\x_n=0}}
\frac{{\bf x}^*{\bf Ax}}{{\bf x}^*{\bf Bx}}=:M.
\]
Since
\[
\Big\{\min_{\substack{{\bf 0}\ne{\bf x}'\in V\\x_n=0}}
\frac{{\bf x}^*{\bf Ax}}{{\bf x}^*{\bf Bx}}\colon 
V\preceq\mathbb{C}^{n-1}, \dim{V}=k\Big\}
\subseteq
\Big\{\min_{{\bf 0}\ne{\bf x}\in U}\frac{{\bf x}^*{\bf Ax}}{{\bf x}^*{\bf Bx}} \colon 
U\preceq\mathbb{C}^n, \dim{U}=k\Big\},         
\]
it follows that
\begin{equation}
\label{m}
M\le\max_{\substack{U\preceq\mathbb{C}^n\\\dim{U}=k}}\min_{{\bf 0}\ne{\bf x}\in U}
\frac{{\bf x}^*{\bf Ax}}{{\bf x}^*{\bf Bx}}.
\end{equation}
Now, by (\ref{mxmn}), (\ref{m}), and~(\ref{maxmin}),
\[
\lambda_k'\le\lambda_k.
\]

To find a reverse inequality, we change the ordering of max and min in the generalized Courant--Fischer theorem:
\[
\lambda_{k+1}=\min_{\substack{U\preceq\mathbb{C}^n\\\dim{U}=n-k}}\max_{{\bf 0}\ne{\bf x}\in U}
\frac{{\bf x}^*{\bf Ax}}{{\bf x}^*{\bf Bx}},\quad k=0,\dots,n-1,
\]
and
\[
\lambda_k'=\min_{\substack{V\preceq\mathbb{C}^{n-1}\\\dim{V}=n-k}}\max_{{\bf 0}\ne{\bf y}\in V}
\frac{{\bf y}^*{\bf A}'{\bf y}}{{\bf y}^*{\bf B}'{\bf y}},\quad k=1,\dots,n-1.
\]
By a simple modification of the previous argument, we obtain
\[
\lambda_k'\ge\lambda_{k+1},
\]
completing the proof.
\end{proof}

\section{LCM--GCD setting on some $T\ne\{1,2,\dots,n\}$}
\label{lcmt}

\subsection{The cases $n=2,3$}

First, let $T=\{u,v\}$, $0<u<v$. Studying $\{u/d,v/d\}$ if $d=\gcd{(u,v)}>1$, we can assume that
$\gcd{(u,v)}=1$. Then
\[
\det{({\bf L}_T-\lambda{\bf G}_T)}=
\left|
\begin{array}{cc}
u-\lambda u&uv-\lambda
\\
uv-\lambda&v-\lambda v
\end{array}
\right|=(uv-1)(\lambda^2-uv),\quad\lambda_1=\sqrt{uv}, \ \lambda_2=-\sqrt{uv}.
\]
Next, let $T=\{1,u,v\}$, where $1<u<v$ and $\gcd{(u,v)}=1$. Then
\begin{align*}
&\det{({\bf L}_T-\lambda{\bf G}_T)}=
\left|
\begin{array}{ccc}
1-\lambda&u-\lambda&v-\lambda
\\
u-\lambda&u-\lambda u&uv-\lambda
\\
v-\lambda&uv-\lambda&v-\lambda v
\end{array}
\right|
\\
&=(u+v-uv-1)\lambda^3+(u+v-uv-1)\lambda^2+(u^2v^2+uv-u^2v-uv^2)\lambda+u^2v^2-u^2v
\\
&-uv^2+uv
=(u-1)(v-1)(\lambda+1)(\lambda^2-uv),\quad
\lambda_1=\sqrt{uv}, \
\lambda_2=-1, \ \lambda_3=-\sqrt{uv}.
\end{align*}
More generally, let $T=\{u,v,w\}$, where $1<u<v<w$ and $\gcd{(u,v)}=\gcd{(u,w)}=\gcd{(v,w)}=1$.
It seems that we do not get pretty results. If, for example, $T=\{2,3,5\}$, then

\begin{equation*}
\det{({\bf L}_T-\lambda{\bf G}_T)}=
\left|
\begin{array}{ccc}
2-2\lambda&6-\lambda&10-\lambda
\\
6-\lambda&3-3\lambda&15-\lambda
\\
10-\lambda&15-\lambda&5-5\lambda
\end{array}
\right|=-22\lambda^3-38\lambda^2+420\lambda+900,
\end{equation*}

\begin{equation*}
    \lambda_1=4.5128, \ \lambda_2=-2.3027, \ \lambda_3=-3.9371.
\end{equation*}

\subsection{The case $T=\{1,p,\dots,p^{n-1}\}$, $p\in\mathbb{P}$}

In this case,
\[
{\bf L}_T=
\left(
\begin{array}{ccccc}
1&p&p^2&\cdots&p^{n-1}
\\
p&p&p^2&\cdots&p^{n-1}
\\
p^2&p^2&p^2&\cdots&p^{n-1}
\\
\vdots&\vdots&\vdots&\vdots&\vdots
\\
p^{n-1}&p^{n-1}&p^{n-1}&\cdots&p^{n-1}
\end{array}
\right)={\bf M}_T,\quad
{\bf G}_T=
\left(
\begin{array}{ccccc}
1&1&1&\cdots&1
\\
1&p&p&\cdots&p
\\
1&p&p^2&\cdots&p^2
\\
\vdots&\vdots&\vdots&\vdots&\vdots
\\
1&p&p^2&\cdots&p^{n-1}
\end{array}
\right)={\bf N}_T.
\]
By Theorem~\ref{mnthm}, the g-eigenvalues of ${\bf L}_T$ to~${\bf G}_T$ are
\[
\lambda_1=p^\frac{n-1}{2}, \ \lambda_2=\cdots=\lambda_{n-1}=-1, \ \lambda_n=-p^\frac{n-1}{2}.
\]

\subsection{Reordering does not matter}

We noted in Remark~\ref{remark} that reordering~$S$ in the MAX--MIN setting only changes (\ref{lambdas})
to~(\ref{lmbds}), so all g-eigenvalues remain.  We now show that reordering~$T$ in the
LCM--GCD setting also keeps the g-eigenvalues. More generally, let ${\bf X}=(x_{ij})$ be a complex square matrix of order~$n$. Given a permutation~$\sigma$ of $(1,\dots,n)$, define
\[
{\bf X}_\sigma=(x^\sigma_{ij}),\quad x^\sigma_{ij}=x_{\sigma(i),\sigma(j)},
\]
and let ${\bf P}_\sigma$ denote the permutation matrix corresponding to~$\sigma$. Since
\[
{\bf X}_\sigma={\bf P}_\sigma{\bf XP}_\sigma\quad{\rm and}\quad\det{{\bf P}_\sigma}=\pm 1,
\]
we have
\[
\det{{\bf X}_\sigma}=\det{{\bf X}},
\]
implying the claim.

\section{Discussion}
\label{discussion}

Above, we first examined the g-eigenvalues of MAX matrices to MIN matrices. The results reveal distinct
structural patterns: the g-eigenvalues are fully characterized and take the form of
one positive value, several values~$-1$ (with multiplicity zero when $n=2$), and one negative value.
This regularity highlights an underlying symmetry and robustness in the generalized eigenstructure of these matrices.

We then turned to the g-eigenvalues of LCM matrices to GCD matrices on
$T=\{1,\dots,n\}$. This topic is 
more intricate. While the above pattern holds for $n\le 4$, it breaks down for $n>4$. In the course of verifying this, we proved a generalization of Cauchy's interlacing theorem for eigenvalues -- namely, the corresponding theorem for g-eigenvalues.

A surprising observation is the emergence of connection to OEIS sequence \seqnum{A004754} in Conjecture~\ref{conj}. If proven, this would provide a novel bridge between matrix theory and number theory, offering a new insight into exploring spectral properties of matrices through binary representations of integers.

We concluded our study by considering sets $T \ne \{1, \dots, n\}$ to demonstrate that certain configurations in the LCM--GCD setting exhibit a MAX--MIN structure. We also showed that, in general, reordering $S$ and~$T$ does not affect the g-eigenvalues, thereby reinforcing the robustness of these matrices under permutations. 

From a computational perspective, determining generalized eigenvalues poses significant challenges, as it typically requires finding the roots of high-degree characteristic polynomials. As $n$ increases, these polynomials become difficult to construct and numerically unstable to solve. However, in constructing the sequence in Conjecture~\ref{conj}, these difficulties can largely be avoided: it is not necessary to form or factorize $p_n(\lambda) = \det(\mathbf{L}_T - \lambda \mathbf{G}_T)$ explicitly. Instead, one can directly compute $p_n(-1) = \det(\mathbf{L}_T + \mathbf{G}_T)$. This approach is computationally lighter, numerically more stable, and sufficient to verify whether $-1$ is a g-eigenvalue.

\bigskip
\hrule
\bigskip

\noindent 2020 {\it Mathematics Subject Classification}:
Primary 15A18; Secondary 11C20.

\noindent \emph{Keywords: }
generalized eigenvalue, Cauchy's interlacing theorem, Hermitian matrix, GCD and LCM matrices, MIN  and MAX matrices.
\bigskip
\hrule
\bigskip

\noindent (Concerned with sequences
\seqnum{A001088}, \seqnum{A003983}, \seqnum{A004754}, 
\seqnum{A051125}, and  \seqnum{A060238}.)

\bigskip
\hrule
\bigskip

\vspace*{+.1in}
\noindent
Received May xx 2025;
revised version received 
Published in {\it Journal of Integer Sequences}, 

\bigskip
\hrule
\bigskip

\noindent
Return to
{Journal of Integer Sequences home page}{https://cs.uwaterloo.ca/journals/JIS/}.
\vskip .1in

\newpage

\begin{appendix}
\setcounter{secnumdepth}{0}
\begin{flushleft}  
    \section{Appendix}
\end{flushleft} 

\begin{lstlisting}[caption=Python code to examine Conjecture \ref{conj}, label=lst:code]
import numpy as np
import math
from scipy.linalg import eig

def gcd_matrix(n):
    """Construct an n x n matrix with entries gcd(i, j)."""
    M = np.zeros((n, n), dtype=int)
    for i in range(1, n+1):
        for j in range(1, n+1):
            M[i-1, j-1] = math.gcd(i, j)
    return M

def lcm_matrix(n):
    """Construct an n x n matrix with entries lcm(i, j)."""
    M = np.zeros((n, n), dtype=int)
    for i in range(1, n+1):
        for j in range(1, n+1):
            M[i-1, j-1] = math.lcm(i, j)
    return M

def find_n_with_minus_one(tol=1e-5, max_n=1000):
    """
    For n = 1 to max_n, computes the generalized eigenvalues for Ax = lambda Bx, where A is the LCM matrix and B is the GCD matrix. Returns a list of n for which -1 appears as a generalized eigenvalue (within a tolerance tol).
    """
    n_list = []
    for n in range(1, max_n+1):
        A = lcm_matrix(n).astype(float)
        B = gcd_matrix(n).astype(float)
        # Compute g-eigenvalues of A to B:
        eigenvalues, _ = eig(A, B)
        # Convert real g-eigenvalues to real values:
        eigenvalues = np.real_if_close(eigenvalues, tol=tol)
        # Check if any g-eigenvalue is equal to -1:
        if any(np.isclose(ev, -1, atol=tol) for ev in eigenvalues):
            n_list.append(n)
    return n_list

if __name__ == '__main__':
    result = find_n_with_minus_one(tol=1e-5, max_n=1000)
    print("Dimensions n for which -1 appears as a g-eigenvalue:")
    print(result)
\end{lstlisting}

\end{appendix}


\begin{thebibliography}{99}

\bibitem{AFKT} M.~Andeli\'c, C.~M. da~Fonseca, C. K\i z\i late\c s, and N. Terzio\u glu, $r$-min and $r$-max matrices with harmonic higher order
Gauss Fibonacci numbers entries, {\em J. Appl. Math. Comput.}~{\bf 71} (2025), 7437--7461. 

\bibitem{ANT} H.~Avron, E.~Ng, and S.~Toledo, A generalized Courant--Fischer minimax
theorem, Technical report, Lawrence Berkeley National Laboratory, 2008.

\bibitem{BL} S.~Beslin and S.~Ligh, Greatest common divisor matrices,
{\em Linear Algebra Appl.}~{\bf 118} (1989), 69--76.

\bibitem{FKT} C.~M. da~Fonseca, C. K\i z\i late\c s, and N. Terzio\u glu, A new generalization of min and
max matrices and their reciprocals counterparts, {\em Filomat}~{\bf 38} (2024), 421--435. 

\bibitem{GKC} B.~Ghojogh, F.~Karray, and M.~Crowley, Eigenvalue and generalized
eigenvalue problems: Tutorial, \texttt{arXiv:1903.11240v3} (2023).

\bibitem{HP} T.~W. Hilberdink and A.~B. Pushnitski, Spectral asymptotics for a family of arithmetical matrices
and connection to Beurling primes, {\em Pure Appl. Funct. Anal}~{\bf 9} (2024), 1145--1161. 

\bibitem{HJ} R.~A.~Horn and C.~R.~Johnson, {\em Matrix Analysis}, Second Edition,
Cambridge Univ. Pr., 2013.

\bibitem{KMR} Y.~Khiar, E.~Mainar, and E.~Royo-Amondarain, Factorizations  and accurate
computations with min and max matrices, {\em Symmetry}~{\bf 17} (2025), Art.~684, 13~pp.

\bibitem{Lo} R.~Loewy, On the smallest singular value in the class of invertible lower triangular (0,1) matrices,
{\em Linear Algebra Appl.}~{\bf 608} (2021), 203--213.

\bibitem{MH} M.~Mattila and P.~Haukkanen, Studying the various properties of MIN and MAX
matrices -- elementary vs. more advanced methods, {\em Spec. Matrices}~{\bf 4} (2016),
101--109.

\bibitem{oeis} {OEIS Foundation Inc.}, The On-Line Encyclopedia of Integer Sequences, 2025.
Published electronically at \url{https://oeis.org}.

\bibitem{RRBA} M.~M.~Rahman, M.~S.~Rahim, M.~N.~A.~S.~Bhutyan, and S.~Ahmed,
Greatest common divisor matrix based phase sequence for PAPR reduction in OFDM
system with low computational overhead, 1st International Conference of
Electrical \& Electronic Engineering (ICEEE), 2015,  97--100.

\bibitem{ZWF} J.~Zhao, C.~Wang, and Y.~Fu, Studying the divisibility of power LCM matrices by power GCD
matrices on gcd-closed sets, {\em J. Combin. Theory, Ser. A,}~{\bf 215} (2025), Art. ID 106063, 27 pp.

\end{thebibliography}
\end{document}